\documentclass[12pt,reqno]{amsart}
\usepackage{amssymb,amsmath,amsthm,amsfonts,mathrsfs}
\usepackage{mathtools}
\usepackage{tikz}
\usepackage[normalem]{ulem}
\usepackage[margin=2.5cm]{geometry}
\usepackage[colorlinks=true,linkcolor=blue,citecolor=blue,urlcolor=blue]{hyperref}
\hypersetup{pdftitle={Global well-posedness for the 2D primitive equations with subcritical horizontal dissipation},pdfauthor={Quyuan Lin and Changhui Tan}}
\title[Fractional Primitive Equations]{Global well-posedness for the 2D primitive equations with subcritical horizontal dissipation}
\author[Quyuan Lin]{Quyuan Lin}
\address[Quyuan Lin]{School of Mathematical and Statistical Sciences, Clemson University, Clemson, SC 29634, USA}
\email{quyuanl@clemson.edu}
\author[Changhui Tan]{Changhui Tan}
\address[Changhui Tan]{Department of Mathematics, University of South Carolina, Columbia SC 29208, USA}
\email{tan@math.sc.edu}

\thanks{{\it Acknowledgment.}  
	 Q. Lin was partially supported by the Simons Foundation (SFI-MPS-TSM-00013384).
	 The work of C. Tan was partially supported by NSF grant DMS-2238219 and a USC-VPR ASPIRE grant.}

\newtheorem{theorem}{Theorem}[section]
\newtheorem{lemma}[theorem]{Lemma}
\newtheorem{proposition}[theorem]{Proposition}

\theoremstyle{definition}
\newtheorem{definition}{Definition}[section]
\theoremstyle{remark}

\def\T{\mathbb T}
\def\R{\mathbb R}
\def\Z{\mathbb Z}
\def\N{\mathbb N}
\def\dd{\mathrm d}
\def\dxdz{\dd x\dd z}
\def\pa{\partial}
\def\Lh{\Lambda_h}
\def\wh{\widetilde w}
\newcommand{\comm}[2]{\left[#1,#2\right]}
\def\K{\mathcal K}
\def\Ks{\K_{\textrm S}}
\def\Ka{\K_{\textrm A}}
\newcommand{\mean}[1]{\langle #1\rangle}

\newcommand{\HH}{\mathsf H}

\begin{document}
\allowdisplaybreaks

\begin{abstract}
We establish global well-posedness of classical solutions to the two-dimensional primitive equations with fractional horizontal dissipation for arbitrarily large initial data in the full subcritical range $1<\alpha\leq2$. Together with the known ill-posedness results for $0\leq\alpha<1$, this establishes the sharp dissipation threshold for large-data global well-posedness in the corresponding solution framework. 

The key ingredient is a hydrostatic energy estimate obtained by splitting the nonlinear energy into symmetric and antisymmetric parts and exploiting the commutator structure of the latter. Using anisotropy and incompressibility, we bound the nonlinear energy by the $L^\infty$ norm of the hydrostatic vorticity times a quadratic velocity norm with only one-half additional horizontal derivative. The vorticity maximum principle then yields enhanced velocity bounds, which close the vorticity estimates and verify the continuation criterion throughout the subcritical regime.
\end{abstract}
\maketitle

\section{Introduction}\label{sec:intro}
The primitive equations describe large-scale oceanic and atmospheric flows under the hydrostatic approximation, in which the vertical momentum equation is replaced by hydrostatic balance. This approximation can be justified as a small aspect ratio limit of the Navier--Stokes equations; see \cite{azerad2001mathematical,li2019primitive,li2022primitive} for rigorous derivations in the viscous setting. Their anisotropic and hydrostatic structure makes horizontal dissipation particularly important, and the contrast between inviscid instability and viscous regularity motivates the study of weaker forms of horizontal dissipation. In this paper, we consider the two-dimensional primitive equations with fractional horizontal dissipation:
\begin{equation}\label{eq:main}
\left\{\begin{aligned}
&\pa_tu+u\pa_xu+\wh\pa_zu+\pa_xp+\nu\Lh^\alpha u=0,\\
&\pa_xu+\pa_z\wh=0,\\
&\pa_zp=0,
\end{aligned}\right.
\end{equation}
on the periodic channel
\[
\Omega\triangleq\{(x,z):x\in\T,\ z\in[0,1]\},\qquad \T=\R/\Z.
\]
Here $(u,\wh)$ is the velocity, $p$ is the pressure, $\nu>0$, and $\Lh^\alpha=(-\pa_x^2)^{\alpha/2}$ with $0<\alpha\leq2$. There is no vertical dissipation. We impose the initial condition $u|_{t=0}=u_0$ and the impermeable boundary conditions
\begin{equation}\label{eq:bc}
\wh(x,0,t)=\wh(x,1,t)=0.
\end{equation}
Incompressibility determines the vertical velocity through
\begin{equation}\label{eq:w}
\wh(x,z,t)=-\int_0^z\pa_xu(x,\zeta,t)\,\dd\zeta.
\end{equation}
Thus recovering $\wh$ costs one horizontal derivative of $u$, which is a central difficulty in controlling the vertical transport. We work in the invariant space
\[
\HH\triangleq\left\{f\in L^2(\Omega):\int_0^1f(x,z)\,\dd z=0\text{ for a.e. }x\in\T\right\},
\]
and denote the hydrostatic vorticity by 
\[
\omega\triangleq\pa_zu.
\]

In the inviscid case $\nu=0$, the system reduces to the hydrostatic Euler equations. The loss of a horizontal derivative in \eqref{eq:w} creates a fundamental obstruction to Sobolev well-posedness. Renardy \cite{renardy2009ill} established linear ill-posedness around suitable shear flows, with unstable modes growing at rates proportional to their horizontal frequency. Han-Kwan and Nguyen \cite{han2016ill} subsequently proved nonlinear ill-posedness, showing that the solution map fails to be H\"older continuous in Sobolev spaces. Later, Ibrahim, Lin, and Titi \cite{ibrahim2021ill} extended these linear and nonlinear ill-posedness results to the system with Coriolis force. Thus, without additional structure on the initial condition, the inviscid problem does not admit the usual local Sobolev theory.

Positive local results are available under structural or analyticity assumptions. Brenier \cite{brenier1999homogeneous} developed a theory for hydrostatic flows with convex vertical velocity profiles. Masmoudi and Wong \cite{masmoudi2012h} established local existence and uniqueness in Sobolev spaces under a local Rayleigh condition, exploiting a cancellation in the vorticity equation to overcome the derivative loss. Kukavica, Masmoudi, Vicol, and Wong \cite{kukavica2014local} allowed multiple regions of monotonicity of the vorticity, with horizontal analyticity in the complementary region. In the analytic setting, Kukavica, Temam, Vicol, and Ziane \cite{kukavica2011local} and Ghoul, Ibrahim, Lin, and Titi \cite{Ghoul2022effect} proved local existence and uniqueness on bounded or periodic domains without a convexity assumption. These results illustrate two ways to recover local well-posedness: impose a favorable structure on the vertical profile, or strengthen the horizontal regularity beyond Sobolev spaces.

Even when local solutions exist, inviscid dynamics can develop singularities. Cao, Ibrahim, Nakanishi, and Titi \cite{cao2015finite} and Wong \cite{wong2015blowup} constructed classes of initial data for which smooth solutions cannot persist globally. Ibrahim, Lin, and Titi \cite{ibrahim2021ill} proved that smooth solutions can blow up in finite time even in the presence of Coriolis force. Collot, Ibrahim, and Lin \cite{collot2024stable} and Ibrahim, Lin, Qian, and Titi \cite{ibrahim2026profile} further identified stable singularity formation mechanisms. The inviscid theory therefore exhibits both instability of the local solution map and finite-time breakdown of smooth solutions.

The viscous theory has a markedly different development. Lions, Temam, and Wang \cite{lions1992new,lions1992equations} laid its mathematical foundations, including global existence of weak solutions for the viscous primitive equations. For the three-dimensional system with full viscosity, Cao and Titi \cite{cao2007global} established global well-posedness of strong solutions for arbitrarily large data. An alternative proof of global strong solvability was given by Kobelkov \cite{kobelkov2006existence}. Kukavica and Ziane \cite{kukavica2007regularity} treated Dirichlet boundary conditions and domains with varying bottom topography. Hieber and Kashiwabara \cite{hieber2016global} subsequently developed an $L^p$ approach based on the hydrostatic Stokes semigroup, extending the global strong solution theory to weaker initial regularity. These works show how the hydrostatic and anisotropic structure can be used to control the nonlinear transport in the presence of viscosity.

A central question is how much of that viscosity is actually needed. Cao, Li, and Titi \cite{cao2016global} proved global well-posedness of the three-dimensional primitive equations with only horizontal viscosity and horizontal thermal diffusion, starting from $H^2$ data. They later extended the theory to initial data near $H^1$, with additional boundedness and vertical regularity assumptions \cite{cao2017strong}, and established global well-posedness with horizontal viscosity and only vertical thermal diffusion \cite{cao2020global}. In particular, vertical viscosity is not necessary for global regularity. These results motivate the further reduction of horizontal viscosity to the fractional dissipation in \eqref{eq:main}.

The direction of the dissipation is essential: vertical viscosity alone does not eliminate the hydrostatic instability identified in \cite{renardy2009ill}. For this partially viscous system, G\'erard-Varet, Masmoudi, and Vicol \cite{gerard2020well} proved local well-posedness for convex initial profiles with Gevrey regularity, while Paicu, Zhang, and Zhang \cite{paicu2020hydrostatic} obtained global well-posedness for small data analytic in the horizontal variable and justified the corresponding hydrostatic limit. Lin, Liu, and Titi \cite{lin2022effect} further studied the effect of fast rotation in this system for initial data analytic in the horizontal variables. Thus the regularizing effect of horizontal dissipation cannot simply be replaced by vertical smoothing.

Fractional horizontal dissipation provides a setting in which to locate the transition between these behaviors. Abdo, Lin, and Tan \cite{abdo2025well} identified $\alpha=1$ as the sharp threshold for local Sobolev well-posedness. In the supercritical range $0<\alpha<1$, the system is ill-posed near suitable shear flows; in the subcritical range $1<\alpha<2$, sufficiently regular data generate unique local solutions. At the critical exponent $\alpha=1$, sufficiently small data yield global well-posedness, whereas ill-posedness occurs near certain large shear flows. The critical behavior therefore depends on the size of the data relative to the dissipation coefficient.

For global regularity, \cite{abdo2025well} established the BKM-type continuation criterion
\begin{equation}\label{eq:BKM}
\int_0^T\|\Lh^{\frac{3-\alpha}{2}}\omega(t)\|_{L^2}^2\,\dd t<\infty.
\end{equation}
When $\alpha\geq3/2$, the required derivative order satisfies $(3-\alpha)/2\leq\alpha/2$, so the basic vorticity energy estimate already controls \eqref{eq:BKM}. This gives global well-posedness in the range $3/2\leq\alpha\leq2$ and raises the question:
\begin{quote}
\emph{Is the system globally well-posed for arbitrarily large initial data throughout the remaining subcritical range $1<\alpha<3/2$?}
\end{quote}

Two improvements narrowed this gap. The work \cite{abdo2025well} first reached $\alpha\geq6/5$ by introducing an energy framework and replacing a vertical Poincar\'e estimate with a sharper interpolation inequality that retains information from both $u$ and $\omega$. This improves the estimate for $\pa_z\omega$ and, in turn, gives the vorticity regularity needed for \eqref{eq:BKM}. Tan and Ye \cite{tan2025global} subsequently extended global well-posedness to $\alpha\geq\alpha_0$, where $\alpha_0\approx1.1108$ is the root in $(1,6/5)$ of $2\alpha_0^3+3\alpha_0^2-4\alpha_0-2=0$. Their estimates make greater use of the maximum principle for $\omega$: interpolation places more of the nonlinear bound in the uniformly controlled $L^\infty$ norm of $\omega$, reducing its dependence on higher Sobolev norms. The resulting enhanced velocity and vorticity estimates lower the threshold for global regularity.

In this paper, we resolve the question throughout the full subcritical range. We establish global well-posedness for every $1<\alpha<3/2$, without a smallness assumption. Together with the known theory for $\alpha\geq3/2$, this gives global well-posedness for all $1<\alpha\leq2$. We state the result for classical solutions in the range $1<\alpha<3/2$, which is the focus of our argument.

Our main result is the following.

\begin{theorem}\label{thm:GWP}
Let $\nu>0$ and $1<\alpha<3/2$. Assume that
\begin{equation}\label{eq:init}
u_0\in\HH,\quad \Lh^\alpha u_0\in L^2(\Omega),\quad
\omega_0\triangleq\pa_zu_0\in L^\infty(\Omega),\quad
\Lh^{\frac\alpha2}\omega_0\in L^2(\Omega),\quad
\pa_z\omega_0\in L^2(\Omega).
\end{equation}
Then \eqref{eq:main}--\eqref{eq:bc} has a unique global classical solution with initial data $u_0$, in the sense of Definition~\ref{def:classical}. For every $T>0$, the solution depends continuously on the initial data in $C([0,T];\HH)$ on sets on which the norms in \eqref{eq:init} are uniformly bounded. %No smallness assumption is imposed on the initial data.
\end{theorem}

Theorem~\ref{thm:GWP} closes the remaining gap between local and global well-posedness and completes the subcritical theory: every dissipation exponent $1<\alpha\leq2$ that guarantees local Sobolev well-posedness also guarantees global well-posedness for arbitrarily large data under the corresponding regularity assumptions. Together with the ill-posedness results of \cite{abdo2025well}, this establishes a sharp dichotomy between global well-posedness in the subcritical range $1<\alpha\leq2$ and ill-posedness in Sobolev spaces near suitable shear flows in the supercritical range $0<\alpha<1$. Thus there is no additional dissipation threshold separating local solvability from global regularity. At the critical exponent $\alpha=1$, small-data global well-posedness coexists with ill-posedness near certain large shear flows, confirming the sharpness of the subcritical condition for unrestricted data size. Figure~\ref{fig:regimes} summarizes how the successive improvements lead to this complete picture.

\begin{figure}[htbp]
\centering
\begin{tikzpicture}[x=1cm,y=1cm,font=\small]
% The horizontal scale is schematic to separate the nearby thresholds.
\definecolor{regimeIP}{RGB}{155,57,48}
\definecolor{regimeWP}{RGB}{38,78,119}
\definecolor{regimeNew}{RGB}{0,103,85}
\draw[->] (0,0.8)--(10.4,0.8) node[right] {$\alpha$};
\foreach \x/\lab in {0/0,3.2/1,4.6/\alpha_0,6/{6/5},7.9/{3/2},10/2}{
  \draw (\x,0.73)--(\x,0.87);
  \node[above] at (\x,0.88) {$\lab$};
}
\foreach \x in {4.6,6,7.9,10}
  \draw[gray!25,densely dotted] (\x,0.65)--(\x,-4.05);
\draw[gray!65,dashed] (3.2,0.65)--(3.2,-4.05);
\node[anchor=east] at (-0.25,0) {Local theory \cite{abdo2025well}};
\node[anchor=east] at (-0.25,-1) {Energy criterion};
\node[anchor=east] at (-0.25,-2) {Abdo-Lin-Tan \cite{abdo2025well}};
\node[anchor=east] at (-0.25,-3) {Tan-Ye \cite{tan2025global}};
\node[anchor=east,font=\small\bfseries] at (-0.25,-4) {This work};
\draw[regimeIP,line width=1.2pt] (0,0)--(3.2,0);
\draw[regimeWP,line width=1.2pt] (3.2,0)--(10,0);
\node[regimeIP,above] at (1.6,0) {Ill-posedness};
\node[regimeWP,above] at (6.6,0) {Local well-posedness};
\draw[gray,dashed] (3.2,-1)--(7.9,-1);
\node[gray,above] at (5.55,-1) {Gap in the global theory};
\draw[regimeWP,line width=1.2pt] (7.9,-1)--(10,-1);
\node[regimeWP,above] at (8.95,-1) {GWP};
\draw[regimeWP,line width=1.2pt] (6,-2)--(10,-2);
\draw[regimeWP,->] (7.9,-1.62)--(6,-1.62);
\node[regimeWP,above] at (8.8,-2) {GWP};
\draw[regimeWP,line width=1.2pt] (4.6,-3)--(10,-3);
\draw[regimeWP,->] (6,-2.62)--(4.6,-2.62);
\node[regimeWP,above] at (8,-3) {GWP};
\draw[regimeIP,line width=1.2pt] (0,-4)--(3.2,-4);
\node[regimeIP,above] at (1.6,-4) {Ill-posedness};
\draw[regimeNew,line width=1.6pt] (3.2,-4)--(10,-4);
\draw[regimeNew,->] (4.6,-3.62)--(3.2,-3.62);
\node[regimeNew,above,font=\small\bfseries] at (7.2,-4) {GWP for every $\alpha>1$};
\foreach \x/\y in {0/0,3.2/0,0/-4,3.2/-4}
  \draw[fill=white,line width=0.8pt] (\x,\y) circle (0.055);
\foreach \x/\y in {10/0,7.9/-1,10/-1,6/-2,10/-2,4.6/-3,10/-3}
  \fill[regimeWP] (\x,\y) circle (0.05);
\fill[regimeNew] (10,-4) circle (0.05);
\end{tikzpicture}
\caption{The dissipation thresholds, shown schematically (not to scale). GWP denotes global well-posedness for arbitrarily large data under the corresponding regularity assumptions. The dashed line at $\alpha=1$ marks the critical case, where small-data global well-posedness coexists with ill-posedness near certain large shear flows.}
\label{fig:regimes}
\end{figure}
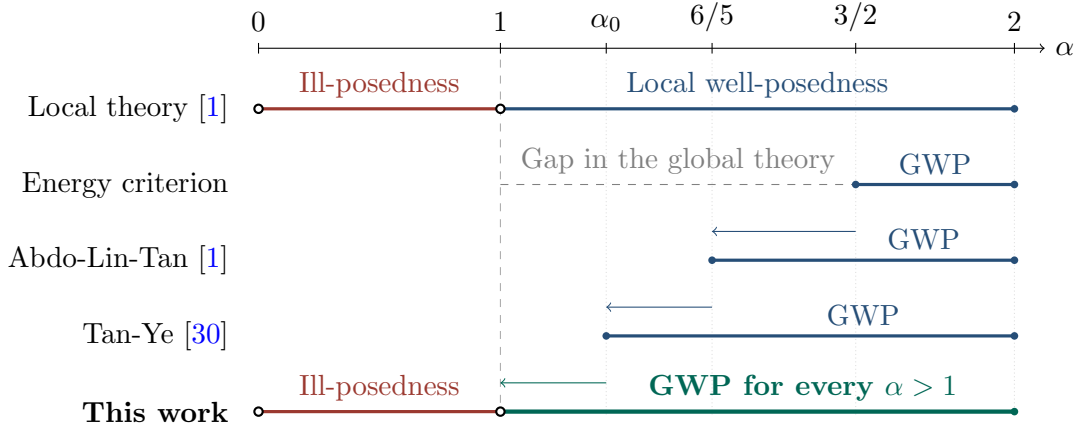

The sharp subcritical global well-posedness theory can be extended by analogous arguments to the smooth and strong solution frameworks developed in \cite{abdo2025well}, under suitable initial regularity assumptions. We do not pursue these extensions here, focusing instead on classical solutions and the sharp dissipation threshold. The initial regularity assumptions in \eqref{eq:init} are stronger than necessary. Determining the weakest assumptions that guarantee global well-posedness is left for future work.

To explain the main estimate, consider the horizontal energy of a smooth solution. Applying $\Lh^s$ to \eqref{eq:main}$_1$, multiplying by $\Lh^su$, and integrating over $\Omega$, we obtain
\[
\frac12\frac{\dd}{\dd t}\|\Lh^su\|_{L^2}^2
+\nu\|\Lh^{s+\frac\alpha2}u\|_{L^2}^2
=-\int_\Omega\Lh^s(u\pa_xu+\wh\omega)\,\Lh^su\,\dxdz.
\]
The pressure term vanishes because $\pa_zp=0$ and $u\in\HH$. Our main new estimate controls the nonlinear integral on the right with a loss of only one-half of a horizontal derivative and a coefficient linear in $\|\omega\|_{L^\infty}$.

\begin{theorem}\label{thm:hydrostatic}
Let $0\leq s<3/2$ and let $u\in C^\infty(\overline\Omega;\R)\cap\HH$. Set $\omega=\pa_zu$ and define $\wh$ by \eqref{eq:w}. Then
\begin{equation}\label{eq:hydrostatic-vort}
\left|\int_\Omega\Lh^s(u\pa_xu+\wh\omega)\,\Lh^su\,\dxdz\right|
\leq C_s\|\omega\|_{L^\infty}\|\Lh^{s+\frac12}u\|_{L^2}^2.
\end{equation}
The constant $C_s$ is independent of $u$.
\end{theorem}

The proof begins by splitting the vertical transport contribution into symmetric and antisymmetric parts with respect to $u$ and $\omega=\pa_zu$. The symmetric part is a vertical derivative; integration by parts and the incompressibility relation $\pa_z\wh=-\pa_xu$ reduce it to a horizontal transport term. The antisymmetric part has the commutator representation
\[
\Ka=-\frac12\int_\Omega\omega\,[\Lh^{2s},\wh]u\,\dxdz.
\]
The commutator gains a horizontal derivative on the lower-frequency factor when $\wh$ has the lower frequency. For the opposite ordering, a further integration by parts in $z$ gives a representation in which the derivatives can again be distributed between two comparable velocity factors. The anisotropic structure and incompressibility are essential in both steps. Together, these cancellations yield \eqref{eq:hydrostatic-vort} for every $0\leq s<3/2$.

The factor $\|\omega\|_{L^\infty}$ appears to the first power in \eqref{eq:hydrostatic-vort}, and the maximum principle bounds it by $\|\omega_0\|_{L^\infty}$ for all time. Thus the nonlinear energy is controlled by a fixed coefficient multiplying a quadratic horizontal Sobolev norm of $u$, without any higher vorticity norm in that coefficient. Since $\alpha/2>1/2$, interpolation absorbs the derivative loss into the dissipation and gives
\[
\frac{\dd}{\dd t}\|\Lh^su\|_{L^2}^2
+\nu\|\Lh^{s+\frac\alpha2}u\|_{L^2}^2
\leq C_{s,\alpha,\nu}\|\omega_0\|_{L^\infty}^{\frac\alpha{\alpha-1}}
\|\Lh^su\|_{L^2}^2.
\]
Consequently, the velocity energy and its dissipation remain bounded at every horizontal index $0\leq s<3/2$ for which the initial energy is finite. This substantially strengthens the velocity estimate in \cite{tan2025global} and removes the restriction on $\alpha$ arising from the iteration of velocity and vorticity bounds. These enhanced estimates allow us to control $\pa_z\omega$ and then verify \eqref{eq:BKM} throughout $1<\alpha<3/2$, completing the global well-posedness argument. The improved velocity control allows us to verify the continuation criterion by vertical interpolation, substantially simplifying the proof for classical solutions.

The rest of the paper is organized as follows. Section~\ref{sec:prelim} collects the basic a priori bounds, the classical solution and continuation theory, and the functional inequalities used in the proof. Section~\ref{sec:commutator} establishes the hydrostatic energy estimate by the symmetric--antisymmetric decomposition. Section~\ref{sec:apriori} derives the enhanced velocity and vorticity bounds, verifies the continuation criterion, and completes the proof of Theorem~\ref{thm:GWP}.

\subsection*{Notations}
We write $L^p=L^p(\Omega)$ for integration over the full domain, while $L_x^p=L^p(\T)$ and $L_z^p=L^p(0,1)$ refer to integration in the indicated variable only, with the other variable held fixed.

For $f=f(x,z)$, we denote its horizontal Fourier coefficients and horizontal mean by
\[
\widehat f_n(z)=\int_{\T}f(x,z)e^{-2\pi inx}\,\dd x,\qquad n\in\Z,
\qquad
\mean{f}(z)=\int_{\T}f(x,z)\,\dd x=\widehat f_0(z).
\]
For $s>0$, the horizontal fractional derivative is defined by
\[
\Lh^sf=\sum_{n\in\Z}(2\pi|n|)^s\widehat f_n(z)e^{2\pi inx}.
\]
We use the convention $\Lh^0=\mathrm{Id}$. Negative homogeneous powers appearing below are applied only to functions with zero horizontal mean, with the sum restricted to $n\neq0$.

The symbol $C$ denotes a positive constant which may vary from line to line. We write $a\lesssim b$ if $a\leq Cb$, and $a\simeq b$ if both $a\lesssim b$ and $b\lesssim a$ hold. Subscripts indicate dependence on the corresponding parameters.

\section{Preliminaries}\label{sec:prelim}
This section collects the preliminary results used below. We refer to \cite{abdo2025well} for further details on the basic a priori estimates and the local and continuation theory.

\subsection{A priori bounds}\label{subsec:basic-bounds}
For a smooth solution, taking the $L^2$ inner product of \eqref{eq:main}$_1$ with $u$ gives
\[
\|u(t)\|_{L^2}^2
+2\nu\int_0^t\|\Lh^{\frac\alpha2}u(\tau)\|_{L^2}^2\,\dd\tau
=\|u_0\|_{L^2}^2.
\]
The transport terms cancel by incompressibility and \eqref{eq:bc}, and the pressure term vanishes because $\pa_zp=0$ and $u\in\HH$.

Differentiating \eqref{eq:main}$_1$ in $z$ and using $\pa_z\wh=-\pa_xu$, we obtain the vorticity equation
\begin{equation}\label{eq:vorticity}
\pa_t\omega+(u\pa_x+\wh\pa_z)\omega+\nu\Lh^\alpha\omega=0.
\end{equation}
Consequently, $\omega$ satisfies the $L^2$ energy identity
\begin{equation}
 \|\omega(t)\|_{L^2}^2 +2\nu\int_0^t\|\Lh^{\frac\alpha2}\omega(\tau)\|_{L^2}^2\,\dd\tau
 =\|\omega_0\|_{L^2}^2,\label{eq:omegaL2}
\end{equation}
and the maximum principle
\begin{equation}
\|\omega(t)\|_{L^\infty} \leq \|\omega_0\|_{L^\infty}.
\label{eq:max}	
\end{equation}
These estimates hold for $0<\alpha\leq2$. For solutions obtained by approximation, the energy identities are used in their inequality form.

The zero vertical mean of $u$ also gives
\[
|u(x,z)| = \left|\int_0^1\bigl(u(x,z)-u(x,\zeta)\bigr)\,\dd\zeta\right|
=\left|\int_0^1\int_\zeta^z\omega(x,r)\,\dd r\,\dd\zeta\right|\leq\|\omega\|_{L^\infty}.
\]
Hence, we have
\begin{equation}\label{eq:uomegainf}
\|u\|_{L^\infty}\leq\|\omega\|_{L^\infty}.
\end{equation}

Finally, \eqref{eq:w} implies $\mean{\wh}=0$.

\subsection{Classical solutions and the regularity criterion}\label{subsec:local-theory}
\begin{definition}[Classical solution]\label{def:classical}
Following \cite[Definition~2]{abdo2025well}, let $T>0$. A function $u$ is a classical solution of \eqref{eq:main}--\eqref{eq:bc} on $[0,T]$ with initial data $u_0$ if it satisfies the system almost everywhere, with $\wh$ given by \eqref{eq:w}, $u(0)=u_0$, and the regularity
\[
\begin{gathered}
u\in C([0,T];\HH),\quad
\Lh^\alpha u\in C([0,T];L^2),\quad
\Lh^{\frac{3\alpha}{2}}u\in L^2(0,T;L^2),\\
\omega=\pa_zu\in C([0,T];L^2)\cap L^\infty(0,T;L^\infty),\quad
\Lh^{\frac\alpha2}\omega\in C([0,T];L^2),\quad
\Lh^\alpha\omega\in L^2(0,T;L^2),\\
\pa_z\omega\in C([0,T];L^2),\quad
\Lh^{\frac\alpha2}\pa_z\omega\in L^2(0,T;L^2).
\end{gathered}
\]
\end{definition}
The bound \eqref{eq:uomegainf} also gives $u\in L^\infty(0,T;L^\infty)$.

\begin{theorem}[Local well-posedness and regularity criterion]\label{thm:LWP-BKM}
Let $\nu>0$ and $1<\alpha<2$. Assume that the initial data satisfy \eqref{eq:init}.
There exists $T_0>0$ such that \eqref{eq:main}--\eqref{eq:bc} has a unique classical solution on $[0,T_0]$ with initial data $u_0$. The solution depends continuously on the initial data in $C([0,T_0];L^2)$ on sets on which the initial norms in \eqref{eq:init} are uniformly bounded, with a common choice of $T_0$.

Moreover, a classical solution on $[0,T)$, with $T<\infty$, extends beyond $T$ as a classical solution provided that \eqref{eq:BKM} holds.
\end{theorem}

Local existence and uniqueness follow from \cite[Theorems~4.7--4.8]{abdo2025well}, and continuous dependence follows from the difference estimate in the proof of \cite[Theorem~4.3]{abdo2025well}. The continuation criterion is a consequence of \cite[Proposition~4.4]{abdo2025well}, which propagates the classical energy under \eqref{eq:BKM}.

For smoother initial data, the same local and continuation theory holds for {\it smooth solutions}; see \cite[Definition~3, Theorem~4.11, and Proposition~4.12]{abdo2025well}. On the other hand, for rougher initial data, \cite[Definition~1 and Theorems~4.2--4.3]{abdo2025well} introduces {\it strong solutions} and establishes global existence together with a conditional uniqueness and continuous-dependence criterion. In this paper, we focus on classical solutions.

When $3/2\leq\alpha<2$, we have $(3-\alpha)/2\leq\alpha/2$, so \eqref{eq:BKM} follows directly from the vorticity energy bound \eqref{eq:omegaL2}. Consequently, classical solutions are globally well-posed in this range; see \cite[Theorem~4.10]{abdo2025well}. Our goal is to verify the same criterion in the remaining range $1<\alpha<3/2$.

\subsection{Fractional Leibniz rule and commutator estimates}\label{subsec:commutator-estimates}
The following estimates act only in the horizontal variable. We apply them at fixed $z$ and then integrate in $z$ using H\"older's inequality.

\begin{lemma}[Fractional Leibniz rule]\label{lem:leibniz}
Let $s>0$, $1<r<\infty$, and $1<p_i,q_i\leq\infty$ satisfy
$1/r=1/p_i+1/q_i$ for $i=1,2$. For smooth periodic functions $f,g$,
\begin{equation}\label{eq:leibniz}
\|\Lh^s(fg)\|_{L_x^r}
\leq C\Bigl(
\|f\|_{L_x^{p_1}}\|\Lh^sg\|_{L_x^{q_1}}
+\|\Lh^sf\|_{L_x^{p_2}}\|g\|_{L_x^{q_2}}
\Bigr).
\end{equation}
\end{lemma}

For a horizontal operator $B$, write $\comm{B}{f}g=B(fg)-fBg$. We have the following Kato--Ponce type commutator estimate \cite{kato1988commutator}. The $L^1$ endpoint case stated below follows from \cite[Theorem~1.2]{li2019kato}.

\begin{lemma}[Commutator estimate]\label{lem:KP}
Let $s>0$. For smooth periodic functions $f,g$,
\begin{equation}\label{eq:KP}
\|\comm{\Lh^s}{f}g\|_{L_x^1}
\leq C_s\Bigl(
\|\pa_xf\|_{L_x^2}\|\Lh^{s-1}g\|_{L_x^2}
+\|\Lh^s f\|_{L_x^2}\|g\|_{L_x^2}
\Bigr).
\end{equation}
When $0<s<1$, assume $\mean{g}=0$ so that the negative homogeneous power is well defined.
\end{lemma}

\subsection{Littlewood--Paley decomposition and Bony paraproducts}\label{subsec:LP}
Let $\{\Delta_j\}_{j\geq0}$ be a smooth Littlewood--Paley decomposition on the nonzero horizontal Fourier modes. We choose real, even symbols supported in $2^{j-1}\leq|n|\leq2^{j+1}$ and forming a partition of unity for $n\neq0$. We write
\[
\begin{aligned}
f =\mean{f}+\sum_{j\geq0}\Delta_jf,\qquad
S_jf = \mean{f}+\sum_{0\leq\ell\leq j}\Delta_\ell f,\qquad
\widetilde\Delta_jf&=\sum_{\substack{\ell\geq0\\|\ell-j|<5}}\Delta_\ell f.
\end{aligned}
\]
For $s\in\R$ and functions with zero horizontal mean,
\begin{equation}\label{eq:LP-hom}
\|\Lh^sf\|_{L^2}^2\simeq\sum_{j\geq0}2^{2js}\|\Delta_jf\|_{L^2}^2.
\end{equation}
For $s>0$, \eqref{eq:LP-hom} also holds without the zero-mean assumption, since both sides ignore the horizontal mean.
We use a fixed gap of five dyadic levels and define
\[
T_fg=\sum_{j\geq0}(S_{j-5}f)\Delta_jg.
\]

The dyadic blocks and low-frequency cutoffs act only in $x$ and commute with $\pa_z$ and the vertical integration in \eqref{eq:w}. The horizontal mean is kept separate from the dyadic blocks. In particular, negative homogeneous powers in \eqref{eq:LP-hom} act only on nonzero modes.

We use the horizontal Bernstein inequalities
\begin{equation}\label{eq:Bernstein}
\begin{aligned}
\|\pa_x^k\Delta_jf\|_{L_x^q}
&\leq C2^{j(k+\frac1p-\frac1q)}\|\Delta_jf\|_{L_x^p},
\qquad k\in\N\cup\{0\},\quad 1\leq p\leq q\leq\infty,\\
\|\Lh^s\Delta_jf\|_{L_x^p}
&\simeq 2^{sj}\|\Delta_jf\|_{L_x^p},
\qquad s\in\R,\quad 1\leq p\leq\infty.
\end{aligned}
\end{equation}
These estimates hold at fixed $z$ and hence also in the corresponding mixed norms. The dyadic blocks and low-frequency cutoffs satisfy
\begin{equation}\label{eq:LP-cutoff}
\|\Delta_j f\|_{L_x^p}+\|S_jf\|_{L_x^p}
\leq C\|f\|_{L_x^p},
\qquad 1\leq p\leq\infty,\quad j\geq0.
\end{equation}
The bound for $S_j$ also holds when $j<0$, since then $S_jf=\mean{f}$.
For a function with horizontal Fourier support in $|n|\leq C_0 2^j$, we also use
\begin{equation}\label{eq:Bernstein-ball}
\|\Lh^s f\|_{L_x^p}\leq C_{s,C_0}2^{sj}\|f\|_{L_x^p},
\qquad s>0,\quad 1\leq p\leq\infty.
\end{equation}
A product of a $\Delta_j$ block and an $S_{j-5}$ cutoff is supported in
$2^{j-2}<|n|<2^{j+2}$. The multiplier $\widetilde\Delta_j$ equals one on this annulus. This also holds when $j<5$, since the low-frequency cutoff then contains only the horizontal mean.

With the above convention, Bony's decomposition reads
\begin{equation}\label{eq:bony}
fg=T_fg+T_gf+R(f,g)+\mean{f}\,\mean{g},
\qquad
R(f,g)=\sum_{j\geq0}\Delta_jf\,\widetilde\Delta_jg.
\end{equation}
Indeed, pairs of nonzero dyadic blocks split into the regions $\ell\leq j-5$, $j\leq\ell-5$, and $|j-\ell|<5$. The paraproducts also include the interactions of a horizontal mean with a nonzero mode, while $\mean{f}\,\mean{g}$ accounts for the two zero modes. The terms $T_fg$, $T_gf$, and $R(f,g)$ represent the low--high, high--low, and comparable-frequency interactions, respectively. Applying \eqref{eq:bony} to $fg$ and $f\Lh^\sigma g$, when $\mean{f}=0$ and $\sigma\geq0$, gives
\begin{equation}\label{eq:bony-commutator}
[\Lh^\sigma,f]g
=[\Lh^\sigma,T_f]g
+\sum_{j\geq0}[\Lh^\sigma,\Delta_jf]S_{j-5}g
+\sum_{j\geq0}[\Lh^\sigma,\Delta_jf]\widetilde\Delta_jg,
\end{equation}
where $[\Lh^\sigma,T_f]g=\Lh^\sigma(T_fg)-T_f(\Lh^\sigma g)$.

\section{The hydrostatic energy estimate}\label{sec:commutator}
Throughout this section, $u\in C^\infty(\overline\Omega;\R)\cap\HH$, $\omega=\pa_zu$, and $\wh$ is given by \eqref{eq:w}. We establish our main estimate \eqref{eq:hydrostatic-vort} for
\[
 \K = \int_\Omega \Lh^s\bigl( u\pa_xu + \wh \omega \bigr)\,\Lh^s u\, \dxdz.
\]

\subsection{The symmetric--antisymmetric decomposition}
We start with the following decomposition:
\[
\omega \Lh^{2s} u = \frac12\bigl( \omega \Lh^{2s} u + u \Lh^{2s} \omega \bigr) + \frac12\bigl( \omega \Lh^{2s} u - u \Lh^{2s} \omega \bigr).
\]
Note that
\[
\omega \Lh^{2s} u + u \Lh^{2s} \omega = \pa_z ( u \Lh^{2s} u ).
\]
Hence, we can rewrite the vertical contribution as
\begin{align*}
 \int_\Omega \Lh^s\bigl(\wh \omega \bigr)\,\Lh^s u\, \dxdz
 & = \int_\Omega \wh\, \omega \,\Lh^{2s} u\, \dxdz\\
 & = \frac12\int_\Omega \wh\, \pa_z ( u \Lh^{2s} u ) \dxdz
  + \frac12\int_\Omega \wh \bigl( \omega \Lh^{2s} u - u \Lh^{2s} \omega \bigr) \dxdz\\
 & = \frac12\int_\Omega \pa_x u\, ( u \Lh^{2s} u ) \dxdz
  + \frac12\int_\Omega \bigl( \wh\, \omega \Lh^{2s} u - \Lh^{2s}(\wh u)\,  \omega \bigr) \dxdz\\
 & = \frac12\int_\Omega \Lh^s\bigl( u\pa_xu \bigr)\,\Lh^s u\, \dxdz
  - \frac12 \int_\Omega \omega\,\comm{\Lh^{2s}}{\wh}u\,\dxdz.
\end{align*}
Therefore, we can decompose $\K$ into a symmetric and an antisymmetric part:
\[
\K = \Ks + \Ka,
\]
where
\begin{align}
 \Ks & = \frac32\int_\Omega \Lh^s\bigl( u\pa_xu \bigr)\,\Lh^s u\, \dxdz,\notag\\
 \Ka & = - \frac12 \int_\Omega \omega\,\comm{\Lh^{2s}}{\wh}u\,\dxdz.\label{eq:Ka}
\end{align}
We estimate the two parts in Propositions~\ref{prop:sym} and~\ref{prop:antisym}, respectively. Combining these results yields Theorem~\ref{thm:hydrostatic}.

\subsection{The symmetric part}
The symmetric part $\Ks$ does not involve the vertical velocity $\wh$. We apply the fractional Leibniz rule \eqref{eq:leibniz} and obtain the following bound.
\begin{proposition}\label{prop:sym}
	For every $s\ge0$, we have
	\begin{equation}\label{eq:sym-bound}
		|\Ks| \leq C_s\|\omega\|_{L^\infty}\|\Lh^{s+\frac12}u\|_{L^2}^2.
	\end{equation}
\end{proposition}
\begin{proof}
Since
\[
 \Ks = \frac32\int_\Omega \Lh^{s-\frac12}\bigl( u\pa_xu \bigr)\,\Lh^{s+\frac12} u\, \dxdz
 = \frac34\int_\Omega \Lh^{s-\frac12} \pa_x \bigl( u^2 \bigr)\,\Lh^{s+\frac12} u\, \dxdz,
\]
we apply H\"older's inequality and the fractional Leibniz rule \eqref{eq:leibniz} and obtain
\begin{align*}
 |\Ks| & \leq  \int_0^1 \|\Lh^{s-\frac12} \pa_x \bigl( u^2 \bigr) \|_{L^2_x} \|\Lh^{s+\frac12} u\|_{L^2_x}\,\dd z\\
 & \leq \int_0^1 \Bigl(C\|u\|_{L^\infty_x}\|\Lh^{s+\frac12} u\|_{L^2_x}\Bigr) \|\Lh^{s+\frac12} u\|_{L^2_x}\,\dd z
 \leq C \|u\|_{L^\infty} \|\Lh^{s+\frac12}u\|_{L^2}^2.
\end{align*}
Using \eqref{eq:uomegainf}, we conclude \eqref{eq:sym-bound}.
\end{proof}

\subsection{The antisymmetric part}
We now estimate the antisymmetric part $\Ka$, which contains the vertical velocity $\wh$. We first explain the role of the commutator structure heuristically for $0<s<3/2$; the case $s=0$ is immediate.

Applying H\"older's inequality directly yields
\[
 |\Ka|\leq \frac12\|\omega\|_{L^\infty}\|\comm{\Lh^{2s}}{\wh} u\|_{L^1}.
\]
To explain the role of the commutator, consider an interaction in which $\wh$ has horizontal frequency $\xi$ much smaller than the frequency $\eta$ of $u$. The commutator symbol satisfies
\[
\bigl||\xi+\eta|^{2s}-|\eta|^{2s}\bigr| \lesssim_s |\xi|\,|\eta|^{2s-1}, \qquad |\xi|\leq\frac12|\eta|.
\]
Since $\wh=-\int_0^z\pa_xu\,\dd\zeta$ and vertical integration is
bounded on $L^2$, the resulting horizontal derivative count is
\[
|\xi|^2|\eta|^{2s-1} = \left(\frac{|\xi|}{|\eta|}\right)^{\frac32-s} |\xi|^{s+\frac12}|\eta|^{s+\frac12}.
\]
Thus the commutator allows both velocity factors to be measured with $s+\frac12$ horizontal derivatives, with an additional decaying frequency ratio when $s<\frac32$. This decay makes the low--high interactions summable and suggests the desired bound $\|\omega\|_{L^\infty}\|\Lh^{s+\frac12}u\|_{L^2}^2$ for this contribution.

However, when $\wh$ carries the higher frequency $(|\xi|\gg|\eta|)$, the commutator no longer provides this gain. We have $\bigl||\xi+\eta|^{2s}-|\eta|^{2s}\bigr| \lesssim |\xi|^{2s}$. The resulting horizontal derivative count $|\xi|^{2s+1}$ cannot be balanced between the high-frequency and low-frequency velocity factors in this bilinear estimate.

To handle the high--low interactions, we return to the trilinear form and obtain an alternative representation of $\Ka$. Integration by parts in $z$, together with
$\pa_z\wh=-\pa_xu$, gives
\begin{align*}
\Ka & = -\frac12\int_\Omega \pa_zu\,\comm{\Lh^{2s}}{\wh}u\,\dxdz
=\frac12\int_\Omega u\,\pa_z\Bigl(\comm{\Lh^{2s}}{\wh}u\Bigr)\,\dxdz\\
& = -\frac12\int_\Omega u\,\comm{\Lh^{2s}}{\pa_xu}u\,\dxdz
+\frac12\int_\Omega u\,\comm{\Lh^{2s}}{\wh}\omega\,\dxdz.
\end{align*}
The first integral vanishes by skew-adjointness after summing all frequency interactions, but we retain both terms when examining the high--low contribution.

In the high--low regime, the coefficient of each commutator carries frequency $\xi$ and its argument carries frequency $\eta$, with $|\xi|\gg|\eta|$. The outer factor $u$ then carries frequency $\kappa=-(\xi+\eta)$, so $|\kappa|\simeq|\xi|$. In both terms, the commutator contributes at most $|\xi|^{2s}$, while $\pa_xu$ and $\wh=-\int_0^z\pa_xu\,\dd\zeta$ each cost one horizontal derivative of the velocity. The resulting derivative count can now be balanced as
\[
|\xi|^{2s+1}
\simeq |\xi|^{s+\frac12}|\kappa|^{s+\frac12},
\]
while the low-frequency argument, either $u$ or $\omega$, can be placed in $L^\infty$. Using \eqref{eq:uomegainf}, this suggests the desired bound $\|\omega\|_{L^\infty}\|\Lh^{s+\frac12}u\|_{L^2}^2$ for this contribution.

The key is that integration by parts replaces the outer high-frequency vorticity by velocity, allowing the derivatives to be shared between two comparable high-frequency velocity factors.

We now make these heuristics precise using Bony's decomposition.
\begin{proposition}\label{prop:antisym}
	For every $0\leq s<3/2$, we have
	\begin{equation}\label{eq:antisym-bound}
		|\Ka| \leq C_s\|\omega\|_{L^\infty}\|\Lh^{s+\frac12}u\|_{L^2}^2.
	\end{equation}
\end{proposition}
\begin{proof}
For $s=0$, the commutator in \eqref{eq:Ka} vanishes. We therefore assume $0<s<3/2$. Set
\[
u_j=\Delta_ju,\qquad
\wh_j=\Delta_j\wh,\qquad
U_j=2^{j(s+\frac12)}\|u_j\|_{L^2},\qquad
\widetilde U_j=2^{j(s+\frac12)}\|\widetilde\Delta_ju\|_{L^2}.
\]
The Littlewood--Paley equivalence \eqref{eq:LP-hom} and finite overlap give
\begin{equation}\label{eq:Usum-hom}
\sum_{j\geq0}\bigl(U_j^2+\widetilde U_j^2\bigr)
\leq C_s\|\Lh^{s+\frac12}u\|_{L^2}^2.
\end{equation}
Here only nonzero horizontal modes occur. Since the horizontal projections commute with vertical integration,
\[
\wh_j=-\int_0^z\pa_xu_j(x,\zeta)\,\dd\zeta,\qquad
\pa_z\wh_j=-\pa_xu_j,\qquad
\wh_j|_{z=0,1}=0.
\]
Consequently, Bernstein's inequality \eqref{eq:Bernstein} and Cauchy--Schwarz in $z$ imply
\begin{equation}\label{eq:blockw-hom}
\|\wh_j\|_{L^2}\leq C2^j\|u_j\|_{L^2},
\qquad
\|\pa_x\wh_j\|_{L^2}\leq C2^{2j}\|u_j\|_{L^2},
\qquad
\|\pa_xu_j\|_{L^2}\leq C2^j\|u_j\|_{L^2}.
\end{equation}

Since $\mean{\wh}=0$, Bony's decomposition \eqref{eq:bony-commutator} gives
\[
\comm{\Lh^{2s}}{\wh}u
=\sum_{j\geq0}\sum_{0\leq\ell\leq j-5} \comm{\Lh^{2s}}{\wh_\ell}u_j
+\sum_{j\geq0}\comm{\Lh^{2s}}{\wh_j}S_{j-5}u
+\sum_{j\geq0}\comm{\Lh^{2s}}{\wh_j}\widetilde\Delta_ju.
\]
We denote the corresponding contributions to $\Ka$ by
$\Ka^{(1)}$, $\Ka^{(2)}$, and $\Ka^{(3)}$, respectively, including the factor $-1/2$ from \eqref{eq:Ka}.

We start with the low--high interactions $\Ka^{(1)}$. Apply the commutator estimate \eqref{eq:KP} at derivative order $2s$, with $f=\wh_\ell$ and $g=u_j$. Integrating the resulting estimate in $z$ and using Cauchy--Schwarz, \eqref{eq:Bernstein}, and \eqref{eq:blockw-hom} yields
\begin{align*}
\|\comm{\Lh^{2s}}{\wh_\ell}u_j\|_{L^1}
&\leq C_s\Bigl(
\|\pa_x\wh_\ell\|_{L^2}\|\Lh^{2s-1}u_j\|_{L^2}
+\|\Lh^{2s}\wh_\ell\|_{L^2}\|u_j\|_{L^2}
\Bigr)\\
&\leq C_s\Bigl(
2^{2\ell}2^{(2s-1)j}+2^{(2s+1)\ell}
\Bigr)\|u_\ell\|_{L^2}\|u_j\|_{L^2}\\
&=C_s\Bigl(
2^{-(\frac32-s)(j-\ell)}
+2^{-(s+\frac12)(j-\ell)}
\Bigr)U_\ell U_j.
\end{align*}
Both exponents in the last line are positive. To sum these bounds, we use Cauchy--Schwarz: for any $a>0$,
\[
\sum_{j\geq0}\sum_{0\leq\ell\leq j-5} 2^{-a(j-\ell)}U_\ell U_j
= \sum_{m\geq 5}2^{-am}\sum_{\ell\geq0}U_\ell U_{\ell+m}
\leq \left(\sum_{m\geq 5}2^{-am}\right)\sum_{\ell\geq0}U_\ell^2
= \frac{2^{-5a}}{1-2^{-a}}\sum_{\ell\geq0}U_\ell^2.
\]
We thus obtain
\begin{equation}\label{eq:K1-hom}
|\Ka^{(1)}|
\leq \frac12\|\omega\|_{L^\infty}\sum_{j\geq0}\sum_{0\leq\ell\leq j-5}\|\comm{\Lh^{2s}}{\wh_\ell}u_j\|_{L^1}
\leq C_s\|\omega\|_{L^\infty}\sum_jU_j^2
\leq C_s\|\omega\|_{L^\infty}\|\Lh^{s+\frac12}u\|_{L^2}^2.
\end{equation}
The summability of the first kernel is the only step requiring $s<3/2$.

Next, we handle the comparable-frequency interactions $\Ka^{(3)}$. The product $\wh_j\widetilde\Delta_ju$ has horizontal Fourier support in $|n|\leq C2^j$. Expanding the commutator and applying \eqref{eq:Bernstein-ball} and \eqref{eq:Bernstein} gives
\begin{align*}
\|\comm{\Lh^{2s}}{\wh_j}\widetilde\Delta_ju\|_{L^1}
&\leq \|\Lh^{2s}(\wh_j\widetilde\Delta_ju)\|_{L^1}
+\|\wh_j\Lh^{2s}\widetilde\Delta_ju\|_{L^1}
\leq C_s2^{2sj}\|\wh_j\|_{L^2}\|\widetilde\Delta_ju\|_{L^2}\\
&\leq C_s2^{(2s+1)j}\|u_j\|_{L^2}\|\widetilde\Delta_ju\|_{L^2}
=C_sU_j\widetilde U_j.
\end{align*}
Note that we do not need to use the commutator structure here.
Hence, an argument similar to that in \eqref{eq:K1-hom} yields
\begin{equation}\label{eq:K3-hom}
|\Ka^{(3)}|
\leq C_s\|\omega\|_{L^\infty}\sum_jU_j\widetilde U_j
\leq C_s\|\omega\|_{L^\infty}\|\Lh^{s+\frac12}u\|_{L^2}^2.
\end{equation}

Finally, we estimate the high--low interactions $\Ka^{(2)}$.
For each $j\geq0$, set
\[
I_j=-\int_\Omega\omega\,\comm{\Lh^{2s}}{\wh_j}S_{j-5}u\,\dxdz.
\]
Thus $\Ka^{(2)}=\frac12\sum_jI_j$. Instead of applying H\"older's inequality directly, we integrate by parts in $z$, using $\pa_zS_{j-5}u=S_{j-5}\omega$, and obtain
\begin{align}
I_j
&=\int_\Omega uF_j\,\dxdz
=\int_\Omega u\,\pa_z\Bigl(\comm{\Lh^{2s}}{\wh_j}S_{j-5}u\Bigr)\,\dxdz\nonumber\\
&=-\int_\Omega u\,\comm{\Lh^{2s}}{\pa_xu_j}S_{j-5}u\,\dxdz
+\int_\Omega u\,\comm{\Lh^{2s}}{\wh_j}S_{j-5}\omega\,\dxdz\nonumber\\
&=-\int_\Omega u\,\Lh^{2s}\bigl((\pa_xu_j)\,S_{j-5}u\bigr)\,\dxdz
+\int_\Omega u\,\Lh^{2s}(\wh_j\,S_{j-5}\omega)\,\dxdz\nonumber\\
&\quad+\int_\Omega u\,(\pa_xu_j)\Lh^{2s}S_{j-5}u\,\dxdz
-\int_\Omega u\,\wh_j\Lh^{2s}S_{j-5}\omega\,\dxdz\nonumber\\
&\triangleq I_{j,1}+I_{j,2}+I_{j,3}+I_{j,4}.
\label{eq:fourterms-hom}
\end{align}
Here $F_j=\pa_z\bigl(\comm{\Lh^{2s}}{\wh_j}S_{j-5}u\bigr)$, and the boundary terms vanish because $\wh_j=0$ at $z=0,1$. We estimate the four terms separately, using the fact that only the horizontal frequencies of the outer velocity $u$ near $2^j$ contribute.
For $j\geq 5$, suppose that $f$ and $g$ have horizontal Fourier supports in
\[
2^{j-1}\leq|n|\leq2^{j+1},
\qquad
|n|\leq2^{j-4},
\]
respectively. The support of the product is contained in the sum of these supports, and hence
\[
\operatorname{supp}_n\widehat{fg}
\subset
\Bigl\{n\in\Z:
2^{j-1}-2^{j-4}\leq|n|
\leq2^{j+1}+2^{j-4}\Bigr\}
\subset\{n\in\Z:2^{j-2}<|n|<2^{j+2}\}.
\]
For $j<5$, the low-frequency factor has only the horizontal zero mode, so the product is still supported in the high-frequency annulus.
Horizontal differentiation and the multipliers $\Lh^{2s}$ do not enlarge Fourier support. Therefore $F_j$ and each of its four summands are supported in $2^{j-2}<|n|<2^{j+2}$. On this set the symbol of $\widetilde\Delta_j$ is identically one, by the partition of unity defining the dyadic blocks. Since $\widetilde\Delta_j$ is also self-adjoint, $F_j$ satisfies the exact identity
\begin{equation}\label{eq:replace}
\int_\Omega uF_j\,\dxdz
=\int_\Omega u\,\widetilde\Delta_jF_j\,\dxdz
=\int_\Omega(\widetilde\Delta_ju)F_j\,\dxdz.
\end{equation}
The same identity holds separately for each summand of $F_j$, permitting the replacement of the outer $u$ by $\widetilde\Delta_ju$ in all four terms of \eqref{eq:fourterms-hom}.

The inequalities \eqref{eq:Bernstein}--\eqref{eq:Bernstein-ball} give
\[
 \|S_{j-5}u\|_{L^\infty}\leq C\|u\|_{L^\infty},\quad
 \|\Lh^{2s}S_{j-5}u\|_{L^\infty}\leq C_s2^{2sj}\|u\|_{L^\infty},\quad
 \|\Lh^{2s}\widetilde\Delta_ju\|_{L^2} \leq C_s2^{2sj}\|\widetilde\Delta_ju\|_{L^2},
\]
and
\[
 \|S_{j-5}\omega\|_{L^\infty}\leq C\|\omega\|_{L^\infty},\quad
 \|\Lh^{2s}S_{j-5}\omega\|_{L^\infty}\leq C_s2^{2sj}\|\omega\|_{L^\infty}.
\]
For the first two terms in \eqref{eq:fourterms-hom}, apply \eqref{eq:replace} and move $\Lh^{2s}$ onto the outer velocity by self-adjointness. Using \eqref{eq:blockw-hom} and \eqref{eq:uomegainf}, we obtain
\begin{align*}
|I_{j,1}|
& = \left|\int_\Omega\Lh^{2s}\widetilde\Delta_j u\,(\pa_xu_j)\,S_{j-5}u\,\dxdz\right|
\leq\|\Lh^{2s}\widetilde\Delta_ju\|_{L^2}
\|\pa_xu_j\|_{L^2}\|S_{j-5}u\|_{L^\infty}\\
&\leq C_s\|u\|_{L^\infty}
2^{(2s+1)j}\|\widetilde\Delta_ju\|_{L^2}\|u_j\|_{L^2}
\leq C_s\|\omega\|_{L^\infty}\widetilde U_jU_j,\\
|I_{j,2}|
& = \left|\int_\Omega \Lh^{2s}\widetilde\Delta_j u\,\wh_j\,S_{j-5}\omega\,\dxdz\right|
\leq\|\Lh^{2s}\widetilde\Delta_ju\|_{L^2}
\|\wh_j\|_{L^2}\|S_{j-5}\omega\|_{L^\infty}\\
&\leq C_s\|\omega\|_{L^\infty}
2^{(2s+1)j}\|\widetilde\Delta_ju\|_{L^2}\|u_j\|_{L^2}
=C_s\|\omega\|_{L^\infty}\widetilde U_jU_j.
\end{align*}
For the other two terms, apply \eqref{eq:replace}, leave $\Lh^{2s}$ on the low-frequency factor, and use \eqref{eq:uomegainf}:
\begin{align*}
|I_{j,3}|
&\leq\|\widetilde\Delta_ju\|_{L^2}
\|\pa_xu_j\|_{L^2}\|\Lh^{2s}S_{j-5}u\|_{L^\infty}
\leq C_s\|u\|_{L^\infty}
2^{(2s+1)j}\|\widetilde\Delta_ju\|_{L^2}\|u_j\|_{L^2}
\leq C_s\|\omega\|_{L^\infty}\widetilde U_jU_j,\\
|I_{j,4}|
&\leq\|\widetilde\Delta_ju\|_{L^2}
\|\wh_j\|_{L^2}\|\Lh^{2s}S_{j-5}\omega\|_{L^\infty}
\leq C_s\|\omega\|_{L^\infty}
2^{(2s+1)j}\|\widetilde\Delta_ju\|_{L^2}\|u_j\|_{L^2}
=C_s\|\omega\|_{L^\infty}\widetilde U_jU_j.
\end{align*}
Summing over $j$ and using \eqref{eq:Usum-hom} yields
\begin{equation}\label{eq:K2-hom}
|\Ka^{(2)}|
\leq C_s\|\omega\|_{L^\infty}\|\Lh^{s+\frac12}u\|_{L^2}^2.
\end{equation}

Combining \eqref{eq:K1-hom}, \eqref{eq:K3-hom}, and \eqref{eq:K2-hom}, we obtain \eqref{eq:antisym-bound}.
\end{proof}

\section{Enhanced a priori bounds}\label{sec:apriori}
We now apply Theorem~\ref{thm:hydrostatic} to derive the estimates needed for continuation when $1<\alpha<3/2$. The velocity estimate below remains valid throughout $1<\alpha\leq2$. We work with a smooth solution on $[0,T]$, with $T$ in its lifespan. The bounds below depend only on the indicated initial norms and the parameters, and hence pass to classical solutions by smooth approximation whenever those initial norms are finite.

\begin{proposition}\label{prop:velocity}
Let $1<\alpha\leq2$ and $0\leq s<3/2$. A smooth solution with $u_0\in\HH$, $\Lh^su_0\in L^2$, and $\omega_0\in L^\infty$ satisfies
\begin{equation}\label{eq:velocitybound}
\sup_{0\leq t\leq T}\|\Lh^su(t)\|_{L^2}^2
+\nu\int_0^T\|\Lh^{s+\frac\alpha2}u(t)\|_{L^2}^2\,\dd t
\leq C_T.
\end{equation}
The constant depends on $\alpha,\nu,s,T,\|\Lh^su_0\|_{L^2}$, and $\|\omega_0\|_{L^\infty}$.
\end{proposition}

\begin{proof}
Taking the $L^2$ inner product of \eqref{eq:main}$_1$ with $\Lh^{2s}u$, we obtain
\[
\frac12\frac{\dd}{\dd t}\|\Lh^su\|_{L^2}^2
+\nu\|\Lh^{s+\frac\alpha2}u\|_{L^2}^2+\K=0.
\]
The pressure term vanishes because $\pa_zp=0$ and $u\in\HH$. By Theorem~\ref{thm:hydrostatic}, \eqref{eq:uomegainf}, and the maximum principle \eqref{eq:max} for $\omega$,
\[
\frac12\frac{\dd}{\dd t}\|\Lh^su\|_{L^2}^2
+\nu\|\Lh^{s+\frac\alpha2}u\|_{L^2}^2
\leq C_s\|\omega_0\|_{L^\infty}\|\Lh^{s+\frac12}u\|_{L^2}^2.
\]
Since $\alpha>1$, interpolation and Young's inequality give
\begin{align*}
C_s\|\omega_0\|_{L^\infty}\|\Lh^{s+\frac12}u\|_{L^2}^2
&\leq C_s\|\omega_0\|_{L^\infty}\|\Lh^su\|_{L^2}^{2-\frac2\alpha}
\|\Lh^{s+\frac\alpha2}u\|_{L^2}^{\frac2\alpha}\\
&\leq\frac\nu2\|\Lh^{s+\frac\alpha2}u\|_{L^2}^2
+C_{s,\alpha,\nu}\|\omega_0\|_{L^\infty}^{\frac\alpha{\alpha-1}}\|\Lh^su\|_{L^2}^2.
\end{align*}
Consequently,
\begin{equation}\label{eq:velocity-energy}
\frac{\dd}{\dd t}\|\Lh^su\|_{L^2}^2
+\nu\|\Lh^{s+\frac\alpha2}u\|_{L^2}^2
\leq C_{s,\alpha,\nu}\|\omega_0\|_{L^\infty}^{\frac\alpha{\alpha-1}}\|\Lh^su\|_{L^2}^2.
\end{equation}
Gr\"onwall's inequality and integration in time yield \eqref{eq:velocitybound}.
\end{proof}

Proposition~\ref{prop:velocity} substantially strengthens the velocity estimate in \cite[Lemma~3.1]{tan2025global}, which gives \eqref{eq:velocitybound} only for
\[
0\leq s\leq\delta_1\triangleq\frac{2\alpha^2-\alpha}{2},
\]
while our estimate holds for every $0\leq s<3/2$ for which the initial velocity norm is finite. Under the classical initial-data assumptions \eqref{eq:init}, this improvement is sufficient to verify the continuation criterion \eqref{eq:BKM} directly, without the iteration used in \cite{tan2025global}. Let us explain the parameter choices and how the argument works.

We first obtain an a priori control of $\|\pa_z\omega\|_{L^2}$. The estimate in the proof of \cite[Lemma~3.2]{tan2025global} requires a uniform velocity bound at the index
\[
\delta_*=\max\left\{\frac\alpha2,\frac{-2\alpha^2+2\alpha+1}{\alpha}\right\}.
\]
For $1<\alpha<3/2$, one checks that $\delta_*\leq1<\alpha<3/2$. Thus the classical initial-data assumptions \eqref{eq:init} and Proposition~\ref{prop:velocity} supply the required bound. Although the cited lemma is stated for $\alpha\in[\alpha_0,6/5)$, its restriction $\alpha\geq\alpha_0$ comes from requiring $\delta_*\leq\delta_1$. Replacing that velocity estimate by Proposition~\ref{prop:velocity}, the same argument applies throughout $1<\alpha<3/2$.
This gives a uniform $L^2$ bound for $\pa_z\omega$ and an $L^2$-in-time bound for $\Lh^{\alpha/2}\pa_z\omega$ on every finite time interval within the lifespan of the solution.

With the bound for $\pa_z\omega$ in hand, the argument of \cite[Lemma~3.3]{tan2025global} relates the available velocity index $\delta$ to the horizontal vorticity index $\rho$ through
\[
0\leq\rho\leq g(\delta)
\triangleq\min\left\{
\frac{\alpha(2\alpha+2\delta-1)}{2(\alpha-2\delta+4)},
\frac{\delta+\alpha-1}{2}
\right\}.
\]
Proposition~\ref{prop:velocity} supplies a velocity bound at any index $0\leq\delta<3/2$ for which $\Lh^\delta u_0\in L^2$. Under \eqref{eq:init}, $\delta=1$ is available and already suffices: for every $1<\alpha<3/2$,
\[
g(1)=\min\left\{
\frac{\alpha(2\alpha+1)}{2(\alpha+2)},\frac\alpha2
\right\}=\frac\alpha2.
\]
Thus the enhanced velocity bound allows the argument of \cite[Lemma~3.3]{tan2025global} to reach $\rho=\alpha/2$ in a single step. The resulting vorticity dissipation controls
\[
\int_0^T\|\Lh^\alpha\omega(t)\|_{L^2}^2\,\dd t,
\]
which is stronger than the bound required by \eqref{eq:BKM}, since $\alpha>(3-\alpha)/2$ for every $\alpha>1$.
This explains why Proposition~\ref{prop:velocity}, combined with the arguments of \cite[Lemmas~3.2--3.3]{tan2025global}, closes the continuation argument throughout $1<\alpha<3/2$, without the iterative improvement of velocity and vorticity indices used there.

We now provide a self-contained proof of the bounds needed for \eqref{eq:BKM}. Thanks to the substantially improved estimate on $u$ in Proposition~\ref{prop:velocity}, the proof simplifies considerably compared with the argument in \cite{tan2025global}.

We first establish the time-integrability of $\|\pa_xu\|_{L^\infty}$. A time weight allows us to obtain this bound using only $\Lh^\alpha u_0\in L^2$ and $\omega_0\in L^\infty$, as supplied by \eqref{eq:init}.

\begin{proposition}\label{prop:ux-integrable}
Let $1<\alpha<3/2$. A smooth solution with $u_0\in\HH$, $\Lh^\alpha u_0\in L^2$, and $\omega_0\in L^\infty$ satisfies
\begin{equation}\label{eq:ux-integrable}
\int_0^T\|\pa_xu(t)\|_{L^\infty}\,\dd t\leq C_T,
\end{equation}
where $C_T$ depends only on $\alpha,\nu,T,\|\Lh^\alpha u_0\|_{L^2}$, and $\|\omega_0\|_{L^\infty}$.
\end{proposition}

\begin{proof}
We first establish the anisotropic interpolation estimate
\begin{equation}\label{eq:ux-anisotropic}
\|\pa_xu\|_{L^\infty}
\leq C_m\|\omega\|_{L^\infty}^{\frac13}
\|\Lh^m u\|_{L^2}^{\frac23},
\qquad m>2.
\end{equation}
For a function $f$ with zero vertical mean, the one-dimensional Gagliardo--Nirenberg interpolation inequality gives
\[
\|f\|_{L_z^\infty}
\leq C\|f\|_{L_z^2}^{\frac23}\|\pa_zf\|_{L_z^\infty}^{\frac13}.
\]
Each $\Delta_j u$ has zero vertical mean. Set $f=\Delta_j u$ and apply \eqref{eq:Bernstein} and \eqref{eq:LP-cutoff} to obtain
\[
\|\pa_x\Delta_j u\|_{L^\infty}
\leq C2^j\|\Delta_j u\|_{L^\infty}\leq C2^j\|\Delta_j u\|_{L_x^\infty L_z^2}^{\frac23} \|\Delta_j\omega\|_{L^\infty}^{\frac13}
\leq C2^{\frac{4j}{3}}\|\Delta_j u\|_{L^2}^{\frac23} \|\omega\|_{L^\infty}^{\frac13},
\]
where 
\[
\|f\|_{L_x^\infty L_z^2}\triangleq\operatorname*{ess\,sup}_{x\in\T}\Bigl(\int_0^1|f(x,z)|^2\,\dd z\Bigr)^{\frac12}.
\]
Since $\mean{\pa_xu}=0$, summing the dyadic blocks and applying H\"older's inequality with exponents $3/2$ and $3$, followed by \eqref{eq:LP-hom}, we obtain
\[
\|\pa_xu\|_{L^\infty}
\leq C\|\omega\|_{L^\infty}^{\frac13} \sum_{j\geq0}2^{\frac{(4-2m)j}3} \bigl(2^{mj}\|\Delta_j u\|_{L^2}\bigr)^{\frac23}
\leq C\|\omega\|_{L^\infty}^{\frac13} \left(\sum_{j\geq0}2^{(2-m)j}\right)^{\frac23} \|\Lh^m u\|_{L^2}^{\frac23}.
\]
The geometric series converges for $m>2$, proving \eqref{eq:ux-anisotropic}.

Fix $2<m<\frac32+\frac\alpha2$ and set $s=m-\frac\alpha2$. Integrating \eqref{eq:ux-anisotropic} in time and using the maximum principle \eqref{eq:max}, we get
\[
 \int_0^T\|\pa_xu(t)\|_{L^\infty}\,\dd t 
 \leq C_m\|\omega_0\|_{L^\infty}^{\frac13}\int_0^T \|\Lh^m u(t)\|_{L^2}^{\frac23}\,\dd t.
\]
If $\Lh^s u_0\in L^2$, H\"older's inequality in time and Proposition~\ref{prop:velocity} give
\[
\int_0^T\|\Lh^m u(t)\|_{L^2}^{\frac23}\,\dd t
\leq T^{\frac23}
\left(\int_0^T\|\Lh^m u(t)\|_{L^2}^2\,\dd t\right)^{\frac13}
\leq C_T.
\]
This proves \eqref{eq:ux-integrable}, but with a constant depending on $\|\Lh^s u_0\|_{L^2}$. When $4/3<\alpha<3/2$, we can choose $2<m\leq3\alpha/2$, so that $s=m-\alpha/2\leq\alpha<3/2$. The initial norm is then controlled by $\|\Lh^\alpha u_0\|_{L^2}$, and the proof is complete in this range.

For $1<\alpha\leq4/3$, however, every $m>2$ gives $s>2-\alpha/2\geq\alpha$, so the preceding argument requires more initial regularity than assumed. To obtain \eqref{eq:ux-integrable} using only $\Lh^\alpha u_0\in L^2$, we introduce a time weight $t^\mu$, with $\mu>0$ to be chosen below. 
Multiplying \eqref{eq:velocity-energy} by $t^\mu$, we obtain
\[
\frac{\dd}{\dd t}\left(t^\mu\|\Lh^s u\|_{L^2}^2\right) +\nu t^\mu\|\Lh^{s+\frac\alpha2} u\|_{L^2}^2
\leq Ct^\mu\|\Lh^s u\|_{L^2}^2
+\mu t^{\mu-1}\|\Lh^s u\|_{L^2}^2.
\]
Here $C$ depends only on $\alpha,\nu,s$, and $\|\omega_0\|_{L^\infty}$. 

Since $\alpha\leq2-\frac\alpha2<s<\frac32<m$, interpolation between $\alpha$ and $m=s+\frac\alpha2$ gives
\[
\|\Lh^s u\|_{L^2} 
\leq\|\Lh^\alpha u\|_{L^2}^{1-\theta}\|\Lh^m u\|_{L^2}^\theta,\qquad \theta = \frac{2(s-\alpha)}{2s-\alpha}.
\]
Thus Young's inequality implies
\[
\mu t^{\mu-1}\|\Lh^s u\|_{L^2}^2
\leq\frac\nu2\, t^\mu\|\Lh^m u\|_{L^2}^2 + Ct^{\mu-\frac1{1-\theta}}\|\Lh^\alpha u\|_{L^2}^2.
\]

Since $s<3/2<2\alpha$, we have
\[
0<\frac{\theta}{1-\theta}=\frac{2(s-\alpha)}{\alpha}<2.
\]
We may therefore choose
\[
\frac{\theta}{1-\theta}<\mu<2.
\]
Fix $m$ and $\mu$ in the above ranges depending only on $\alpha$. Combining the weighted energy inequality with Young's inequality and using Proposition~\ref{prop:velocity} at index $\alpha$, we obtain
\[
\frac{\dd}{\dd t}\left(t^\mu\|\Lh^s u\|_{L^2}^2\right)
+\frac\nu2\,t^\mu\|\Lh^m u\|_{L^2}^2
\leq Ct^\mu\|\Lh^s u\|_{L^2}^2
+C_Tt^{\mu-\frac1{1-\theta}}.
\]
The last term is integrable at $t=0$ because $\mu-\frac1{1-\theta}>-1$. Moreover, $t^\mu\|\Lh^s u(t)\|_{L^2}^2$ vanishes at $t=0$ for each smooth solution. Gr\"onwall's inequality and integration therefore yield
\[
\sup_{0<t\leq T}t^\mu\|\Lh^s u(t)\|_{L^2}^2
+\nu\int_0^Tt^\mu\|\Lh^m u(t)\|_{L^2}^2\,\dd t
\leq C_T,
\]
where $C_T$ depends only on the quantities stated in the proposition, and not on $\|\Lh^s u_0\|_{L^2}$.

Finally, H\"older's inequality in time gives
\begin{align*}
\int_0^T\|\Lh^m u(t)\|_{L^2}^{\frac23}\,\dd t
&=\int_0^T
\left(t^\mu\|\Lh^m u(t)\|_{L^2}^2\right)^{\frac13}
t^{-\frac\mu3}\,\dd t\\
&\leq
\left(\int_0^Tt^\mu\|\Lh^m u(t)\|_{L^2}^2\,\dd t\right)^{\frac13}
\left(\int_0^Tt^{-\frac\mu2}\,\dd t\right)^{\frac23}
\leq C_T,
\end{align*}
since $\mu<2$. Together with \eqref{eq:ux-anisotropic} and the maximum principle \eqref{eq:max}, this proves \eqref{eq:ux-integrable}.
\end{proof}

The same argument proves Proposition~\ref{prop:ux-integrable} under the weaker initial velocity assumption $\Lh^\delta u_0\in L^2$ whenever $\delta\geq0$ and $\delta>2-3\alpha/2$. Thus $\delta=0$ is admissible when $\alpha>4/3$, while $\delta=1/2$ works throughout $1<\alpha<3/2$. We use the stronger assumption $\Lh^\alpha u_0\in L^2$ to match the classical solution framework.

Let $q=\pa_z\omega$. Differentiating \eqref{eq:vorticity} in $z$ and using $\pa_z\wh=-\pa_xu$ yields
\begin{equation}\label{eq:q-dynamics}
\pa_tq+(u\pa_x+\wh\pa_z)q+\nu\Lh^\alpha q
=(\pa_xu)q-\omega\pa_x\omega.
\end{equation}
We establish the $L^2$ energy estimate on $q$ in the following proposition. This gives the bound obtained in \cite[Lemma~3.2]{tan2025global} throughout $1<\alpha<3/2$, under the classical initial-data assumptions, with a simpler proof based on Proposition~\ref{prop:ux-integrable}.

\begin{proposition}\label{prop:vertical-vorticity}
Let $1<\alpha<3/2$. A smooth solution with $u_0\in\HH$, $\Lh^\alpha u_0\in L^2$, $\omega_0\in L^\infty$, and $q_0=\pa_z\omega_0\in L^2$ satisfies
\begin{equation}\label{eq:qbound-short}
\sup_{0\leq t\leq T}\|q(t)\|_{L^2}^2
+\nu\int_0^T\|\Lh^{\frac\alpha2}q(t)\|_{L^2}^2\,\dd t
\leq C_T.
\end{equation}
The constant depends on $\alpha,\nu,T,\|\Lh^\alpha u_0\|_{L^2}$, $\|\omega_0\|_{L^\infty}$, and $\|q_0\|_{L^2}$.
\end{proposition}

\begin{proof}
The transport term in \eqref{eq:q-dynamics} vanishes in the $L^2$ energy balance, by incompressibility and the boundary conditions. Thus
\[
\frac12\frac{\dd}{\dd t}\|q\|_{L^2}^2 +\nu\|\Lh^{\frac\alpha2}q\|_{L^2}^2
=-\int_\Omega\omega\pa_x\omega\,q\,\dxdz +\int_\Omega(\pa_xu)q^2\,\dxdz
\triangleq J_1+J_2.
\]
For $J_1$, write $\omega\pa_x\omega=\frac12\pa_x(\omega^2)$ and distribute one-half horizontal derivative to each factor. The fractional Leibniz rule \eqref{eq:leibniz} and the maximum principle \eqref{eq:max} give
\begin{align*}
|J_1|
&\leq C\|\Lh^{1/2}(\omega^2)\|_{L^2} \|\Lh^{1/2}q\|_{L^2}
\leq C\|\omega_0\|_{L^\infty} \|\Lh^{1/2}\omega\|_{L^2}\|\Lh^{1/2}q\|_{L^2}\\
&\leq\frac\nu2\|\Lh^{\frac\alpha2}q\|_{L^2}^2
+C_{\alpha,\nu}\|\omega_0\|_{L^\infty}^2 \|\Lh^{\frac\alpha2}\omega\|_{L^2}^2.
\end{align*}
For $J_2$, we simply use
\[
|J_2|\leq\|\pa_xu\|_{L^\infty}\|q\|_{L^2}^2.
\]
Consequently,
\begin{equation}\label{eq:qineq-short}
\frac{\dd}{\dd t}\|q\|_{L^2}^2+\nu\|\Lh^{\frac\alpha2}q\|_{L^2}^2
\leq 2\|\pa_xu\|_{L^\infty}\|q\|_{L^2}^2
 +C_{\alpha,\nu}\|\omega_0\|_{L^\infty}^2 \|\Lh^{\frac\alpha2}\omega\|_{L^2}^2.
\end{equation}
Gr\"onwall's inequality applied to \eqref{eq:qineq-short} gives, for $0\leq t\leq T$,
\begin{align*}
&\|q(t)\|_{L^2}^2+\nu\int_0^t\|\Lh^{\frac\alpha2}q(\tau)\|_{L^2}^2\,\dd\tau\\
&\quad\leq\left(\|q_0\|_{L^2}^2
+C_{\alpha,\nu}\|\omega_0\|_{L^\infty}^2
\int_0^t\|\Lh^{\frac\alpha2}\omega(\tau)\|_{L^2}^2\,\dd\tau\right)
\exp\left(2\int_0^t\|\pa_xu(\tau)\|_{L^\infty}\,\dd\tau\right)
\leq C_T,
\end{align*}
where we have used \eqref{eq:omegaL2} and Proposition~\ref{prop:ux-integrable}. This proves \eqref{eq:qbound-short}.
\end{proof}

We can now verify the continuation criterion \eqref{eq:BKM} by vertical interpolation.

\begin{proposition}\label{prop:continuation-interpolation}
Under the same assumptions as in Proposition~\ref{prop:vertical-vorticity}, a smooth solution satisfies
\begin{equation}\label{eq:continuation-bound}
\int_0^T\|\Lh^{\frac{3-\alpha}{2}}\omega(t)\|_{L^2}^2\,\dd t
\leq C_T.
\end{equation}
The constant depends only on the quantities listed in Proposition~\ref{prop:vertical-vorticity}.
\end{proposition}

\begin{proof}
The one-dimensional interpolation inequality
\[
\|\pa_z f\|_{L_z^2}^2
\leq C\|f\|_{L_z^2}\|\pa_z^2f\|_{L_z^2}
+C\|f\|_{L_z^2}^2
\]
requires no boundary conditions on $f$. Set $f=\Delta_j u$ at each fixed $x$. Integration in $x$ and Cauchy--Schwarz give
\begin{align*}
\|\Delta_j\omega\|_{L^2}^2
&\leq C\int_{\T}
\|\Delta_j u(x,\cdot)\|_{L_z^2}
\|\Delta_j q(x,\cdot)\|_{L_z^2}\,\dd x
+C\|\Delta_j u\|_{L^2}^2\\
&\leq C\|\Delta_j u\|_{L^2}\|\Delta_j q\|_{L^2}
+C\|\Delta_j u\|_{L^2}^2.
\end{align*}
Using \eqref{eq:LP-hom}, the identity $3-\alpha=(3-3\alpha/2)+\alpha/2$, and Cauchy--Schwarz in the dyadic sum, we obtain
\[
\begin{aligned}
\|\Lh^{\frac{3-\alpha}{2}}\omega\|_{L^2}^2
&\leq C\sum_{j\geq0}2^{(3-\alpha)j}\|\Delta_j\omega\|_{L^2}^2\\
&\leq C\sum_{j\geq0}
\bigl(2^{(3-\frac{3\alpha}{2})j}\|\Delta_j u\|_{L^2}\bigr)
\bigl(2^{\frac{\alpha j}{2}}\|\Delta_j q\|_{L^2}\bigr)
+C\sum_{j\geq0}2^{(3-\alpha)j}\|\Delta_j u\|_{L^2}^2\\
&\leq C
\left(\sum_{j\geq0}2^{(6-3\alpha)j}\|\Delta_j u\|_{L^2}^2\right)^{\frac12}
\left(\sum_{j\geq0}2^{\alpha j}\|\Delta_j q\|_{L^2}^2\right)^{\frac12}
+C\|\Lh^{\frac{3-\alpha}{2}} u\|_{L^2}^2\\
&\leq C\|\Lh^{3-\frac{3\alpha}{2}}u\|_{L^2}
\|\Lh^{\frac\alpha2}q\|_{L^2}
+C\|\Lh^{\frac{3-\alpha}{2}} u\|_{L^2}^2.
\end{aligned}
\]
As $\frac{3-\alpha}2\leq3-\frac{3\alpha}2$, Young's inequality yields
\[
\|\Lh^{\frac{3-\alpha}{2}}\omega\|_{L^2}^2
\leq C\left(
\|\Lh^{3-\frac{3\alpha}{2}}u\|_{L^2}^2
+\|\Lh^{\frac\alpha2}q\|_{L^2}^2
\right).
\]
Since $0<3-3\alpha/2<3\alpha/2$ for $1<\alpha<3/2$, Proposition~\ref{prop:velocity} at index $\alpha$ gives
\[
\int_0^T\|\Lh^{3-\frac{3\alpha}{2}}u(t)\|_{L^2}^2\,\dd t
\leq C\int_0^T\|\Lh^{\frac{3\alpha}{2}}u(t)\|_{L^2}^2\,\dd t
\leq C_T.
\]
The $q$ term is integrable by \eqref{eq:qbound-short}. Integrating the preceding pointwise estimate over $[0,T]$ therefore proves \eqref{eq:continuation-bound}.
\end{proof}

Proposition~\ref{prop:continuation-interpolation} replaces the direct evolution estimate for the horizontal Sobolev norms of $\omega$ in \cite[Lemma~3.3]{tan2025global} with a vertical interpolation argument. This substantially simplifies the proof of the continuation bound. All the initial norms required by Propositions~\ref{prop:vertical-vorticity} and~\ref{prop:continuation-interpolation} are included in the classical initial-data assumptions \eqref{eq:init}. For rougher initial data in the strong solution framework, however, this interpolation argument can require more horizontal regularity of $u_0$ than an argument that directly propagates the horizontal regularity of $\omega$. The pursuit of global well-posedness under the lowest possible initial regularity is left for future work.

Finally, Proposition~\ref{prop:continuation-interpolation}, together with Theorem~\ref{thm:LWP-BKM}, yields the global well-posedness asserted in Theorem~\ref{thm:GWP}. Indeed, start with the unique local classical solution supplied by Theorem~\ref{thm:LWP-BKM}. The estimates above, justified for classical solutions by smooth approximation, bound \eqref{eq:BKM} uniformly up to any finite endpoint of its lifespan. The continuation criterion therefore rules out a finite maximal existence time. Uniqueness and continuous dependence follow from the same classical solution theory.

\medskip
\subsection*{Declaration on the use of generative AI}
During the preparation of this manuscript, the authors used ChatGPT Sol 5.6 and Astra 6 (OpenAI) and Claude Opus 5.5 (Anthropic) as auxiliary tools for language refinement, LaTeX formatting, and proof auditing. The manuscript was written by the authors, who independently analyzed and verified all mathematical content and arguments. The authors take full responsibility for the accuracy and originality of the work.

\end{document}